\documentclass[12pt,reqno]{amsart}
 \usepackage[dvips]{epsfig}
 \usepackage{amsgen, amstext,amsbsy,amsopn, amsthm, amsfonts,amssymb,amscd,amsmat
 h,euscript,enumerate,url,verbatim,calc,xypic}
 \usepackage{latexsym}
 \usepackage{graphics}
 \usepackage{color}

 \newcommand{\n}{\mathfrak{n} }
 \newcommand{\m}{\mathfrak{m} }

 \newcommand{\Q}{\mathbb{Q}}
 \newcommand{\Z}{\mathbb{Z}}

 \newcommand{\ov}{\overline}

\newcommand{\Ac}{\mbox{${\mathcal A}$}}
\newcommand{\Bc}{\mbox{${\mathcal B}$}}
\newcommand{\Cc}{\mbox{${\mathcal C}$}}

    \newcommand{\coker}{\operatorname{coker}}

  \newcommand{\Ass}{\operatorname{Ass}}

 \newcommand{\dm}{\operatorname{dim}}
 \newcommand{\h}{\operatorname{ht}}

 \newcommand{\Spec}{\operatorname{Spec}}

 \newcommand{\depth}{\operatorname{depth}}

 \newcommand{\lm}{\lambda}

\newcommand{\sbs}{\subseteq}
\newcommand{\ra}{\rightarrow}

\newcommand{\iso}{\cong} % For Isomorphism

\newcommand{\proset}{\,\mathrel{\lower 4pt\hbox{$\scriptscriptstyle/$}
\mkern -14mu\subseteq }\,} %for proper subset

 \newtheorem{theorem}{Theorem}[section]
 \newtheorem{corollary}[theorem]{Corollary}
 \newtheorem{lemma}[theorem]{Lemma}
 \newtheorem{proposition}[theorem]{Proposition}

 \newtheorem{conjecture}[theorem]{Conjecture}

\usepackage{amsmath}

 \theoremstyle{definition}

 \newtheorem{definition}[theorem]{Definition}

 \newtheorem{example}[theorem]{Example}

 \title[Chern number of filtrations ]
{On Some Conjectures about the Chern Numbers of Filtrations}

\author{Mousumi Mandal,  Balwant Singh and J. K. Verma}
\thanks{The first author is supported by the National Board for Higher Mathematics, India.}

\address{Department of Mathematics, Indian Institute of Technology 
Bombay, Powai, Mumbai 400076, India} \email{mousumi@math.iitb.ac.in}

\address{Department of Mathematics, Indian Institute of Technology 
Bombay, Powai, Mumbai 400076  and MU-DAE Centre for Excellence in Basic Sciences, Kalina Campus, Santacruz, Mumbai 400098, India}
\email{balwantbagga@gmail.com}

\address{Department of Mathematics, Indian Institute of Technology 
Bombay, Powai, Mumbai 400076, India} \email{jkv@math.iitb.ac.in}

\begin{document}

 \maketitle
 \begin{abstract}
 Let $I$ be an $\m$-primary ideal of a Noetherian local ring $(R,\m)$ of positive dimension.  The coefficient $e_1(\mathcal A)$ of the Hilbert polynomial of an $I$-admissible filtration $\mathcal A$ is called the Chern number of $\mathcal A.$
The Positivity Conjecture of Vasconcelos for the Chern number of the integral closure filtration $\{\ov{I^n}\}$ is proved for a 
 $2$-dimensional complete local domain and more generally for  any analytically unramified local ring $R$ whose integral closure in its total ring of fractions is  Cohen-Macaulay as an $R$-module. It is proved that if $I$ is a parameter ideal then the Chern number
of the $I$-adic filtration is non-negative.
Several other results on the Chern number of the integral closure filtration are established, especially in the case when $R$ is not necessarily Cohen-Macaulay. 

 \end{abstract}

\section*{Introduction}

For a nonzero polynomial $P=P(X)\in \mathbb Q[X]$ of degree $d$ such that $P(n)\in \mathbb Z$ 
for $n\gg 0,$ it is customary to write $P$ in the form  
$$ 
P= \sum_{i=0}^{d}(-1)^i e_i(P) \binom{X+d-i}{d-i} 
$$ 
with $e_i(P)$ integers, called the Hilbert coefficients of $P.$ The top two Hilbert coefficients have special names: $e_0(P)$ is the multiplicity of $P$ and $e_1(P),$ the subject matter of this paper,  is the {\bf Chern number} of $P.$  

If $I$ is an $\m$-primary ideal of a Noetherian local ring $(R,\m)$ of positive dimension and $P_I$ is the polynomial associated to the function $n\mapsto \lm(R/I^{n+1}),$ where $\lm$ denotes length as $R$-module, then the Hilbert coefficients  $e_i(P_I)$ are called the Hilbert coefficients of $I$ and are also denoted by $e_i(I).$ In particular, $e_1(I)$ is the Chern number of $I.$ 

If $\mathcal A=\{\mathcal A_n\}_{n\geq 0} $ and $\mathcal B=\{\mathcal B_n\}_{n\geq 0} $ are (decreasing) filtrations of ideals of a ring $R$ then the {\bf admissibility} of  $\Ac$ over $\Bc$ means that there exists a nonnegative integer $k$ such that $\Ac_{n+k} \subseteq  \Bc_n \subseteq \Ac_n$ for every $n\geq 0.$  We say 
that $\Ac$ is $I$-admissible, where $I$ is an ideal of $R,$ if $\Ac$ is admissible over the $I$-adic filtration. 

For an ideal $I$ of a ring $R,$ the integral closure
of $I,$ denoted $\overline{I},$ is the ideal of $R$ consisting of all elements of $R$ which are integral over $I,$ i.e. elements $a \in R$ satisfying an equation of the form
$ a^r+b_1a^{r-1}+ \cdots + b_r=0 $
with  $r$ some positive integer and  $b_i \in I^i$ for every $i.$ Applying this construction to the powers $I^n$ of an $\m$-primary ideal $I$ in a Noetherian local ring $(R,\m),$ we get the filtration 
$\{\ov{I^n}\}$ on $R.$  If $R$ is analytically unramified then this filtration is $I$-admissible by  Rees \cite {rees}. It follows that  the normal Hilbert function of $I,$ namely the function $n\mapsto  \lm(R/\ov{I^{n+1})},$ is given, for $n\gg 0,$ by a  polynomial $\overline{P}_I,$  called the {\bf normal Hilbert polynomial} of $I.$ The Hilbert coefficients 
$e_i(\overline{P}_I)$  are called the normal Hilbert coefficients of $I$ and are also denoted by $\ov{e}_i(I).$ 
In particular, $\ov{e}_1(I)$ is the {\bf normal Chern number} of $I.$ 

At a conference held in 2008 in Yokohama, Japan, Wolmer Vasconcelos \cite{vas} announced several conjectures about the the Chern number of a parameter ideal and the normal Chern number of an $\m$-primary ideal in a Noetherian local ring $(R,\m).$  

In this paper, we discuss two of these conjectures, namely the Positivity Conjecture and the Negativity Conjecture. We also provide some general estimates on the Chern number.

{\bf The Positivity Conjecture of Vasconcelos} says that if  $I$ is an $\m$-primary ideal of  an analytically unramified Noetherian local ring $(R,\m)$ of positive dimension then $\ov{e}_1(I) \geq 0.$

We settle this conjecture for an analytically unramified Noetherian local ring $(R,\m)$ whose integral closure in its total ring of fractions
is  Cohen-Macaulay as an $R$-module. This is done in section 1. A  consequence is that the Positivity Conjecture holds for a 2-dimensional complete Noetherian local domain. We also settle the conjecture in case there is a Cohen-Macaulay local ring $(S,\n)$ dominating  $(R,\m)$ such that $\lm(S/R)$ is finite.

We  show in section 2  that there is a $2$-dimensional analytically unramified
Noetherian local ring constructed from a 1-dimensional simplicial complex for which the normal Chern number 
is negative. This simplicial complex is non-pure.  On the other hand, we show that the normal Chern number
of the maximal homogeneous ideal of the face ring of  a simplicial complex $\Delta$ of dimension $d-1$ is $df_{d-1}-f_{d-2},$  where $f_i$ is the number of $i$-dimensional faces of $\Delta.$ This implies that if $\Delta$ is pure then $\ov{e}_1(\m) \geq 0.$ These results indicate perhaps that for the Positivity Conjecture to hold,   the ring needs to be quasi-unmixed, i.e. its completion $\hat R$ should be equidimensional.

Recall here that $R$ is said to be unmixed if $\dim \hat R/\mathfrak p=\dim \hat R$  for every $\mathfrak p\in \Ass \hat R.$

{\bf The Negativity Conjecture of Vasconcelos} says that if $J$ is a parameter ideal of an unmixed Noetherian local ring $R$ of positive dimension then $e_1(J) < 0$ if and only if $R$ is not Cohen-Macaulay.

Vasconcelos \cite{vas} settled the conjecture for a domain  that is essentially of finite type over a field.  It was settled for a universally catenary Noetherian local domain containing a field by Ghezzi, Hong and Vasconcelos in \cite{ghv}. 
They also proved that if $S$ is a Cohen-Macaulay local ring and $\mathfrak p$ is a prime ideal of $S$ such that  $\dim S/\mathfrak p \geq 2$ and $S/\mathfrak p$ is not Cohen-Macaulay then $e_1(J) < 0$ for every  parameter ideal $J$ of $S/\mathfrak p.$
Mandal and Verma \cite{mv} settled the Negativity Conjecture for parameter ideals in certain quotients of  a regular local ring.
The conjecture has been settled recently by Ghezzi, Goto, Hong, Ozeki, Phuong and Vasconcelos \cite{gghopv}.

In section 3, we discuss the corresponding question for a finite module $M$ (of positive dimension)  over a Noetherian local ring $(R,\m)$ with respect to an ideal $I$ such that $\lm(M/IM)  < \infty.$ In this case, if $P_I(M,X)$ is the polynomial associated to the function  $n\mapsto \lambda(M/I^{n+1}M),$  we write $e_i(I,M)$  for $e_i(P_I(M,X)).$ In particular, we have the coefficient $e_1(I,M),$ which we call the Chern number of $I$ with respect to $M.$ While the multiplicity $e_0(I,M)$ has been  studied extensively, the investigation of the Chern number $e_1(I,M),$ especially over non-Cohen-Macaulay rings, has begun only recently. We show that if $J$ is a parameter ideal with respect to $M$ then  
$e_1(J,M)\leq  0$ and, further, that $e_1(J,M)< 0$ if  $\depth M = \dim R-1.$ We also show that if $R$ is   Cohen-Macaulay and $M$ is an unmixed $R$-module with $\dim M= \dim R$ then $M$ is Cohen-Macaulay if and only if $e_1(J,M)=0,$ for one (resp. every)  parameter ideal $J.$ 
 
In   section 4, we determine some bounds for the normal Chern number of  an $\m$-primary ideal in terms of a minimal reduction $J$ of $I.$ Using Serre's formula for multiplicity of a parameter ideal in terms of the Euler characteristic
of the Koszul homology, we show that
$$ 
\ov e_1(J)\leq \sum_{n\geq 1}\lm(\ov{J^n}/J\ov{J^{n-1}})+e_1(J). 
$$
This generalizes a formula of Huckaba and Marley \cite{hm} for the integral
closure filtration in a Cohen-Macaulay local ring.

In the  final  section  5, we find some estimates on the Chern number of a prameter ideal $J$ in a Noetherian local ring  $(R,\m)$ assuming  that there exists  a Cohen-Macaulay local ring $(S,\n)$ dominating  $(R,\m)$ with  $\lm (S/R)< \infty.$  We show in this case that $\mu_R(S/R)\, \leq \,  -e_1(J) \,\leq \, \lm (S/R),$ 
and that if the equalities hold for every parameter ideal $J$ then  $R$ is Buchsbaum.

\noindent {\bf Acknowledgements:} We thank Shiro Goto for useful discussions.

\section{The Positivity Conjecture of Vasconcelos}

\begin{conjecture}[{\bf The Positivity Conjecture of Vasconcelos}] \label{pc}   Let $I$ be an $\m $-primary ideal of  an  analytically unramified Noetherian local ring $(R,\m )$ of positive dimension. Then 
$\ov{e}_1(I) \geq 0.$ 
\end{conjecture}

In this section, we prove that the conjecture holds for a ring $R$ which satisfies any one of the following conditions: (i) $R$ is  Cohen-Macaulay; (ii) the integral closure of $R$ is Cohen-Macaulay as an $R$-module; (iii) $R$ is a complete local domain of dimension 2; (iv) some other technical conditions. See Corollary 
\ref{cor-pc} for details. 

Let the notation and assumptions be as in the conjecture. 

Put $\Ac_n=\ov{I^n},$ the integral closure of $I^n$ in $R.$ Then, the filtration $\Ac=\{\Ac_n\}$ is the integral closure filtration of the $I$-adic filtration and, as noted in the Introduction, 
the analytical unramifiedness  of $R$ implies by \cite{rees}  that $\Ac$ is $I$-admissible.   
 More generally, let  $\Bc=\{\Bc_n\}$ be any filtration of ideals of $R$ which is $I$-admissible.  Then 
the function $n\mapsto  \lm(R/\Bc_{n+1})$ is given, for $n\gg 0,$ by a  polynomial $P_{\Bc}\in \mathbb Q[X].$ 
In this situation, we write $e_i(\Bc)$ for $e_i(P_{\Bc}).$ In particular, $e_i(\Ac) = \ov{e}_i(I).$ 

By a {\bf finite cover} $S/R,$ we mean a ring extension $R\subseteq S$  such that $S$ is a finite $R$-module. Then $S$ is a Noetherian semilocal ring.  We say that the finite cover $S/R$ is {\bf birational} if $R$ is reduced and $S$ is contained in the total quotient ring of $R;$ that $S/R$ is of {\bf finite length} if $\lm (S/R)$ is finite;  and that $S/R$ is  {\bf Cohen-Macaulay} if $S$ is Cohen-Macaulay as an $R$-module.   

\begin{theorem} \label{bb} Let $(R,\m)$ be a Noetherian local ring of dimension $d\geq 1.$ 
Let $S/R$ be a finite cover such that at least one of the following two conditions holds: {\rm (i)} $S/R$ is of finite length;  or {\rm (ii)} $S/R$ is birational.  Let $I$ be an $\m$-primary ideal of $R,$ and let $\Bc$ be a filtration  of $R$ such that $\Bc$ is $I$-admissible and $R\cap I^nS \sbs \Bc_n$ for $n\gg 0.$  Then $e_1(\Bc) \geq e_1(I, S).$  
\end{theorem}
 
\begin{proof}
Let $\Cc$ denote the filtration of $R$ given by $\Cc_n=R\cap I^nS.$ For our proof, we need four length functions and their associated polynomials in $\Q[X]$ as listed in the following table:
 
 \begin{center}
\begin{tabular}{|l|c|} \hline 
Length function        & Associated polynomial\\ 
\hline
 $\lm (R/I^{n+1})$    & $P_I= P_I(X)$ \\ 
 $\lm (S/I^{n+1}S)$      & $P_{I,S}=P_{I,S}(X)$ \\
 $\lm (R/\mathcal B_{n+1})$      & $P_{\mathcal B}= P_{\mathcal B}(X)$ \\
 $\lm (R/\Cc_{n+1})$     & $P_{\mathcal C} = P_{\mathcal C}(X)$ \\ \hline
 \end{tabular}             
\end{center}

\noindent By the given conditions on $\Bc,$ there exists a nonnegaive integer $k$ such that  
$$ 
\Cc_{n+k} \subseteq \mathcal B_{n+k }\subseteq I^{n}\subseteq \Cc_{n} \sbs \Bc_n  
$$
for $n \geq  0.$ Therefore 
$$
\lm (R/\Cc_{n+k})\geq  \lm (R/\Bc_{n+k})\geq \lm (R/I^{n})\geq \lm (R/\Cc_{n}) \geq \lm (R/\Bc_{n}) 
$$
for $n\geq  0,$ from which it follows that 
$$
d = \deg P_I =\deg P_{\mathcal B}  =\deg P_{\mathcal C}   \mbox{ and } e_0(P_I )=e_0(P_{\Bc} ) = e_0(P_{\Cc} ). \eqno{(A)} 
$$
Now, the inequalities $\lm (R/\Cc_{n}) \geq \lm (R/\Bc_{n})$ for $n\geq  0$ imply that 
$$ 
e_1(P_{\Bc})\geq e_1(P_{\Cc}).  \eqno{(B)} 
$$ 

Assume now that (i) holds, i.e. $S/R$ is of finite length (but may not be birational).  Then, for $n\gg 0,$  we have $I^nS\sbs R,$ so $\Cc_n=I^nS.$  
 Therefore, for $n\gg 0,$ we have  
$$   
\lm (R/\Cc_n) = \lm (R/I^{n}S) = \lm (S/I^{n}S) - \nu,  
$$
where $\nu= \lm (S/R).$ Consequently, $P_{\Cc} = P_{I,S}-\nu.$  Now, by $(B),$ we get 
$$ 
e_1(P_{\Bc})\geq  e_1(P_{\Cc})= e_1(P_{I,S} -\nu).  \eqno{(C)} 
$$ 
If $d=1$ then $e_1(P_{I,S} -\nu) = e_1(P_{I,S}) +\nu \geq  e_1(P_{I,S}),$ while if $d\geq 2$ then $e_1(P_{I,S} -\nu)= e_1(P_{I,S}).$ In either case, 
$e_1(P_{I,S} -\nu) \geq e_1(P_{I,S}).$ Therefore, by $(C),$  we get  
$$e_1(\Bc)= e_1(P_{\Bc}) \geq 
 e_1(P_{I,S}) =  e_1(I,S),$$ 
which proves  the assertion under condition (i). 

Now, drop the assumption (i)  and assume (ii), so that $S/R$ is birational (but may not be of finite length).  
 Then $S/R$ is annihilated by a nonzero divisor of $R,$ so $\dim S/R \leq d-1.$ 
Therefore, since $\deg P_{\Cc} = d $ by (A), the exact sequence 
$$
0\ra R/\Cc_n \ra S/I^{n}S \ra S/(R + I^{n}S) \ra 0
$$
shows that
$\deg(P_{I,S})= \deg P_{\Cc} $ and $e_0(P_{I,S})= e_0(P_{\Cc}) .$ Combining this with the inequalities $P_{I,S}(n)   \geq   P_{\Cc}(n)$ for $n\gg 0,$ which also result from the exact sequence,  we get 
$e_1(P_{\Cc})  \geq e_1(P_{I,S}).$ Thus, using $(B)$ again, we get  
$$e_1(\Bc)=e_1(P_{\Bc}) \geq e_1(P_{\Cc}) \geq e_1(P_{I,S})= e_1(I,S).$$ This proves the assertion under condition (ii).  
\end{proof}

\begin{corollary} \label{cor-pc}  Let $(R,\m)$ be an anlytically unramified Noetherian local ring of positive dimension.  Then the Positivity Conjecture \ref{pc}  holds for $R$ if $R$ satisfies any one of the following conditions: 

{\rm  (1)} $R$ has a finite Cohen-Macaulay cover which is of finite length or is birational. 

{\rm (2) }$R$ is Cohen-Macaulay {\rm (cf. \cite{hm}).}  

{\rm (3)} $\dim R=1.$ 

{\rm (4)} The integral closure of $R$ is Cohen-Macaulay as a an $R$-module. 

{\rm (5)} $\dim R=2$ and all maximal ideals of the integral closure of $R$ have the same height.

{\rm (6)} $R$ is a complete local integral domain of dimension 2.          

\end{corollary}
\begin{proof} Since a minimal reduction of an $\m$-primary ideal $I$ gives rise to the same integral closure filtration as  $I$ does, it is enough to prove the conjecture (under any of the above conditions on $R$) for a parameter ideal of $R.$  So, let $I$ be a parameter ideal of $R,$ and let $\Ac$ be the integral closure filtration of the $I$-adic filtration of $R.$ Then,  as noted earlier,  $\Ac$ is $I$-admissible, and we have $\ov e_1(I)= e_1(\Ac).$  Thus we have to show that $e_1(\Ac)\geq 0$ under each of the six condtions. 

(1) Let $S/R$ be a finite  Cohen-Macaulay cover which is of finite length or is birational. Since $S$ is integral over $R,$ we have $R\cap I^nS\sbs \Ac_n$ for every $n\geq 0$ by Proposition 1.6.1 of \cite{HS}. So, by the above theorem applied with $\Ac$ in place of $\Bc,$ we get  $e_1(\Ac)\geq e_1(I,S).$ Since $S$ is Cohen-Macaulay as an $R$-module and $I$ is a parameter ideal of $R,$ we have $e_1(I,S)=0.$ Thus 
$e_1(\Ac)\geq 0.$ 

(2) Apply (1) to the trivial cover $R/R.$

(3) Since $R$ is reduced and one dimensional, it is  Cohen-Macaulay, so we can use (2).

For the remaining part of the proof, let $R'$ be the integral closure of $R$ in its total quotient ring. Then $R'/R$ is a finite birational cover by \cite{rees}, and $\dim R'=\dim R.$  

(4) Since $R'/R$ is a finite birational cover which is Cohen-Macaulay, we are done by (1). 

(5) $\dim R'=2$ implies that $R'$ is Cohen-Macaulay as a ring. Now, it is easy to see that the assumption that all maximal ideals of $R'$ have the same height implies that $R'$ is Cohen-Macaulay as an $R$-module. So the assertion follows from (4).  

(6) In this case, it is well known that $R'$ is local, so (5) applies.  
\end{proof}

\section{The Positivity Conjecture for the Maximal Homogeneous Ideal of a Face Ring}

In this section, we show that the Positivity Conjecture holds for the filtration $\overline{\m^n}$ where $\m$ is the maximal homogeneous ideal of the face ring of a pure simplicial
complex $\Delta.$ Let $\Delta$ be a $(d-1)-$dimensional simplicial complex. Let $f_i$ denote the number of $i$-dimensional faces of
$\Delta$ for $i=-1,0, \ldots, d-1.$ Here $f_{-1}=1.$ Let $\Delta$ have $n$ vertices $\{v_1, v_2, \ldots, v_n\}.$ Let $x_1,x_2, \ldots,x_n$ be indeterminates over a field $k.$ The ideal $I_{\Delta}$ of $\Delta$ is the ideal generated by the square free monomials
$x_{a_1}x_{a_2}\ldots x_{a_m}$ where $1 \leq a_1 < a_2 < \cdots < a_m \leq n$ and $\{v_{a_1},v_{a_2}, \ldots,v_{a_m}\} \notin \Delta.$   The face ring of $\Delta$ over a field $k$ is defined as  $k[\Delta]=k[x_1,x_2,\ldots,x_n]/I_{\Delta}.$

\begin{lemma}
 Let $R$ be a Noetherian ring and $I$ be an ideal of $R$ such that the associated graded ring $G(I)=\bigoplus_{n=0}^\infty I^n/I^{n+1}$ is reduced. Then $\ov{I^n}=I^n$ for all $n$.
\end{lemma}

\begin{proof}
 Let $\mathcal R(I)=\oplus_{n \in \Z} I^nt^n$ denote the extended Rees ring of $I.$  Since $G(I)=\mathcal R(I)/(u)$ where $u=t^{-1},$  and $G(I)$ is reduced, $(u)=P_1\cap P_2 \cap \ldots \cap P_r$ for some height one prime ideals $P_1, \ldots, P_r$ of $\mathcal R(I).$ Therefore
$(u)$ is integrally closed in $\mathcal R(I).$ As  $P_i\mathcal R(I)_{P_i}=(u)\mathcal R(I)_{P_i}$ for all $i,$ 
$\mathcal R(I)_{P_i}$ is a DVR for all $i.$ Since $u$ is  
regular, $\text{Ass}(\mathcal R(I)/(u^n))=\{P_1, P_2, \ldots, P_r\}$
for all  $n\geq 1.$ Thus $(u^n)=\cap_{i=1}^rP_i^{(n)}$ is integrally closed. Hence  ${I^n}=(u^n)\cap R$  is integrally closed for all $n.$
\end{proof}

\begin{lemma}\label{chernface}
 Let $\Delta$ be a $(d-1)-$dimensional simplicial complex.
Let $\m$ denote the maximal homogeneous ideal of the face ring
$k[\Delta]$ over a field $k.$ Then $\m^n=\ov{\m}^n$ for all $n.$ and 
$$e_1(\m)=\ov{e}_1(\m)=df_{d-1}-f_{d-2}.$$
\end{lemma}

\begin{proof}
 
Since $k[\Delta]$ is standard graded $k-$algebra, $G(\m)=k[\Delta].$ Hence $G(\m)$ is reduced and consequently
$\m$ is a normal ideal. Moreover, $\lm(\m^n/\m^{n+1})=\dim_kk[\Delta]_n.$ 
The Hilbert Series of the face ring is written as  $$H(k[\Delta],t)=\frac{h_0+h_1t+\cdots+h_st^s}{(1-t)^d}.$$ Put $h(t)=h_0+h_1t+\cdots + h_st^s$ where 
the face vector $(f_1, f_0, \ldots,f_{d-1})$ and the $h$-vector are related by the equation
$$\sum_{i=0}^sh_it^i=\sum_{i=0}^df_{i-1}t^i(1-t)^{(d-i)}$$
by \cite[Lemma 5.1.8]{bh}.
Then by \cite[Proposition 4.1.9]{bh} we have 
$$e_1(\m)=h'(1)=df_{d-1}-f_{d-2}.$$ 
\end{proof}

\begin{theorem}
 Let $\Delta$ be a  pure simplicial complex. Then $$\ov{e}_1(\m)=e_1(\m)\geq 0.$$
\end{theorem}
\begin{proof}
 Let $\dim \Delta =d-1.$ We prove that if $\Delta$ is a pure simplicial complex then $df_{d-1}\geq f_{d-2}$. Let $\sigma$ be a facet. For any $v_i\in\sigma=\{v_1,\ldots , v_d\}$, $\sigma\setminus \{v_i\}$ is a $(d-2)-$dimensional face and $\sigma\setminus \{v_i\}$ are distinct for all $i=1,\ldots ,d$. Therefore each facet gives rise to $d$, $(d-2)$-dimensional faces. But different facets may produce same faces of dimension $d-2$. Since $\Delta$ is pure each $(d-2)-$dimensional face is contained in a facet. Hence $df_{d-1}\geq f_{d-2}$. Therefore
$\ov{e}_1(\m)\geq 0$ by Lemma \ref{chernface} 
\end{proof}

\begin{example}
The above theorem indicates that the the maximal homogeneous ideal of the  face ring of a non-pure simplicial complexe may have negative Chern number. Indeed, consider the simplicial complex 
$\Delta_n$ on the vertices $\{v_1,v_2, \ldots v_{n+2}\}$ where $n \geq 2$  and $$\Delta_n=\{\{v_1,v_2\},v_3, \ldots,v_{n+2} \}.$$
Then $e_1(\m)=df_{d-1}-f_{d-2}=-n.$
Hence we need to add the assumption of quasi-unmixedness on the ring in Vasaconccelos' Positivity conjecture.
\end{example}

 \section{The Negativity Conjecture of Vasconcelos}
 
In this section we show that the Chern number of any parameter
ideal with respect to a finite module over a Noetherian local ring is non-negative. For this purpose, we need to generalize a result of Goto-Nishida \cite[Lemma 2.4]{gn} to  modules.
 \begin{proposition}
\label{p1}
 Let $(R,\m)$ be a Noetherian local ring and let $M$ be a finite $R$-module with $\dm M=1$. If $a$ is a parameter for $M$ then $$e_1((a),M)=-\lm(H^0_{\m}(M)).$$
\end{proposition}

\begin{proof}
Let $N=H^0_{\m}(M)$ and $\ov M=M/N$. Notice that $H^0_{\m}(\ov M)=0$ and $\dm \ov M=\dm M=1$, which implies $\depth \ov M=1$. Thus $\ov M$ is Cohen-Macaulay $R$-module. 
 Consider the exact sequence 
 $$0\longrightarrow N\longrightarrow M\longrightarrow \ov M\longrightarrow 0.$$ 
By taking tensor product with $R/(a)^n$ we get the  exact sequence for all $n\geq 1$
\begin{equation}\label{eq}
 0\longrightarrow \ker \phi_n \longrightarrow N/a^nN\xrightarrow{\phi_n} M/a^nM\longrightarrow \ov M/a^n\ov M\longrightarrow 0.
 \end{equation}
%  For all large $n$ we have $(a)^nN=(a)^nH^0_{\m}(M)\subseteq \m^nH^0_{\m}(M)=0.$ Moreover, $\ker \phi_n= a^nM\cap N/a^nN=0$ for all large $n.$ 
By Artin-Rees Lemma, there is a $k$ such that $$a^nM\cap N= a^{n-k}(a^kM\cap N)\subseteq a^{n-k} N\subseteq \m^{n-k}N=0$$ for large $n$. Hence $\ker \phi_n =0$ for all large $n$.
Thus, for all large $n,$  we get the  exact sequence:
 $$0\longrightarrow N\longrightarrow M/{a}^nM\longrightarrow \ov M/{a}^n\ov M\longrightarrow 0.$$ 
 Hence we have $\lm(N)+\lm(\ov M/ a^n\ov M)=\lm(M/ a^nM)$. Since $\ov M$ is Cohen-Macaulay, $$\lm(\ov M/ a^n\ov M)=e_0( (a^n),\ov M)=e_0((a),\ov M)n=e_0((a),M)n.$$ For large $n$, $\lm(M/ a^nM)=ne_0((a),M)-e_1((a),M)$. Therefore $$e_1((a),M)=-\lm(H^0_{\m}(M)).$$
\end{proof}

\begin{corollary}
 Let $(R,\m)$ be a Noetherian local ring and $M$ be a finite $R$-module with $\dm M=1$. Let $a$ be a parameter for $M$. Then $e_1((a),M)=0$ if and only if $M$ is a Cohen-Macaulay module.
\end{corollary}
% \begin{proof}
%  Let $M$ be not Cohen-Macaulay. Then $e_1((a),M)=-\lm(H^0_{\m}(M))< 0$. 
% The Converse is well known. See \cite[Theorem 1.1.8]{bh}. 
% \end{proof}

In order to investigate the Chern number
for finite modules of dimension $d \geq 2$ we use induction on dimension. The principal tool for this purpose is the concept of superficial element of an ideal with respect to a module.
The next theorem  is found in Nagata \cite[22.6]{nag} for Noetherian local rings.
It is proved for modules over Noetherian local rings in \cite{m}.

 \noindent
{\bf Nagata's Theorem:} Let $(A,\m)$ be a Noetherian local ring and $M$ be a finite $A$-module with $\dm M=d\geq 2$. Let $I$ be an ideal of definition of $M$ and let $a$ be a superficial element for $I$ with respect to $M$. Set $\ov M=M/aM$. Then
 %\begin{enumerate}
$$ P_{\ov I}(\ov M,n)=\bigtriangleup P_I(M,n)+\lm(0:_Ma).$$
In particular,  
  \begin{equation*}
e_i(\ov I,\ov M)=\left \{\begin{array}{llll}
e_i(I,M) & \mbox{if $0\leq i< d-1$}.\\
e_{d-1}(I,M)+(-1)^{d-1}\lm(0:_Ma) & \mbox{if $i=d-1$}.
\end{array} \right.
\end{equation*}
% \end{enumerate}

\begin{lemma}\label{ll1}
 Let $(R,\m)$ be a Noetherian local ring and $M$ be a finite $R$-module with $\dm M=d$. Let $I$ be an ideal of definition for $M$ generated by  ${\bf x}=x_1,\ldots, x_d$ which is a  superficial sequence for $I$  with respect to $M$. If $M$ is not Cohen-Macaulay then $M/x_1M$ is not Cohen-Macaulay.
\end{lemma}
\begin{proof}
 Suppose $\ov M=M/x_1M$ is Cohen-Macaulay. Then $\ov{x_2},\ldots ,\ov{x_d}$ is an $\ov M$-regular sequence. Thus $\lm(\ov M/(\ov{x_2},\ldots \ov{x_d})\ov M)=e_0(\ov{x_2},\ldots \ov{x_d},\ov M).$ Since $x_1$ is superficial for $M$, $e_0({\bf x},M)=e_0(\ov{x_2},\ldots ,\ov{x_d},M)$. Hence $\lm(M/({\bf x})M)=e_0({\bf x},M)$. Therefore $M$ is Cohen-Macaulay which is  a contradiction.
\end{proof}

\begin{proposition}
 Let $(R,\m)$ be a Noetherian local ring and $M$ be a finite $R$-module with $\dm M=d$ and $\depth M=d-1$. Let $J$ be generated by a system of parameters for $M$. Then $e_1(J,M)<0$. 
\end{proposition}
\begin{proof}
 Apply induction on $d$. The $d=1$ case is already done. Suppose $d=2$. Let $J=(a,b)$. We may assume that $(a,b)$ is a superficial sequence for $J$ with respect to $M$ and since $\depth M=1$, $a$ is $M$-regular. Let $\ov M=M/aM$. Then $\dm \ov M=1$. By Nagata's Theorem, we have $e_1(\ov J,\ov M)=e_1 (J,M).$ By Lemma \ref{ll1}, $\ov M$ is not Cohen-Macaulay. Thus 
  $e_1(\ov J,\ov M)<0$. Therefore $e_1(J,M)<0$. 

Next assume that $d\geq 3$ and $J=(x_1,\ldots ,x_d)$  where $x_1,\ldots ,x_d$ is a superficial sequence with respect to $M$. Let $\ov M=M/x_1M$ then $\dm \ov M=d-1$. By Nagata's Theorem  we get $e_1(\ov J,\ov M)=e_1(J,M)$. If $M$ is not Cohen-Macaulay then $\ov M$ is also not Cohen-Macaulay and hence by induction hypothesis $e_1(\ov J,\ov M)<0$, which implies $e_1(J,M)<0$. 
 \end{proof}

\begin{theorem}
 Let $(R,\m)$ be a Noetherian local ring and $M$ be a finite $R$-module
  with $\dm M=d$. Let $J$ be an ideal generated by  a system of parameters for $M$. Then $e_1(J,M)\leq 0.$
\end{theorem}
\begin{proof}
Apply induction on $d$. The $d=1$ case is already proved. Suppose $d=2$.
Let $J=(x,y)$  where $x,y$ is a superficial sequence for $J$ with respect to $M$. 
 Consider the exact sequence
 $$0\longrightarrow M/(0:_Mx)\stackrel{x}\longrightarrow M\longrightarrow M/xM\longrightarrow 0.$$
Applying $H^0_{\m}(.)$  we get 
\begin{equation}\label{ses1}
0\longrightarrow H^0_{\m}(M/(0:_Mx))\stackrel{x}\longrightarrow H^0_{\m}(M)\stackrel{g}\longrightarrow H^0_{\m}(M/xM)\longrightarrow C\longrightarrow 0
\end{equation}
where $C=\coker g$.
Consider the exact sequence 
$$0\longrightarrow (0:_Mx)\longrightarrow M\longrightarrow M/(0:_Mx)\longrightarrow 0.$$
Applying $H^0_{\m}(.)$ on the exact sequence we get
$$0\longrightarrow H^0_{\m}(0:_Mx)\longrightarrow H^0_{\m}(M) \longrightarrow H^0_{\m}(M/(0:_Mx))\longrightarrow 0.$$
Since $H^0_{\m}(0:_Mx)=0:_Mx$, we have $$\lm(0:_Mx)=\lm(H^0_{\m}(M))-\lm(H^0_{\m}(M/(0:_Mx))).$$ 
Subtracting $\lm(H^0_{\m}(M/xM))$ from both sides of the above equation we get 

%\begin{eqnarray*}
$$\lm(0:_Mx)-\lm(H^0_{\m}(M/xM))
=
\lm(H^0_{\m}(M))-\lm(H^0_{\m}(M/xM))\\
-\lm(H^0_{\m}(M/(0:_Mx))).
$$
%\end{eqnarray*}

\noindent
From the exact sequence (\ref{ses1}) we get $$\lm(H^0_{\m}(M/(0:_Mx)))-\lm(H^0_{\m}(M))+\lm(H^0_{\m}(M/xM)=\lm(C).$$ Therefore we have $\lm(0:_Mx)-\lm(H^0_{\m}(M/xM)=-\lm(C)$. By Theorem \ref{2}, we get 
$$e_1(\ov J,\ov M)=e_1(J,M)-\lm(0:_Mx).$$ 
By Proposition \ref{p1}, $e_1(\ov J,\ov M)=-\lm(H^0_{\m}(M/xM))$. Therefore $$e_1(J,M)=\lm(0:_Mx)-\lm(H^0_{\m}(M/xM))=-\lm(C)\leq 0.$$
Let $d\geq 3$ and $a\in J$ be a  superficial for $J$ with respect to $M$. Since  $e_1(J,M)=e_1(J/(a),M/aM),$ we are done by induction.
\end{proof}

\begin{proposition}
 Let $(R,\m)$ be a Noetherian local ring and $M$ be a finite $R$-module with $\dm M=d\geq 2$. Let $J$ be a parameter for $M$. If $M/H^0_{\m}(M)$ is Cohen-Macaulay then $e_1(J,M)=0$.
\end{proposition}
\begin{proof}
 Let $W=H^0_{\m}(M)$ and $\ov M=M/W$. 
 %Consider the following exact sequence
 %$$0\longrightarrow W \longrightarrow R \longrightarrow R/W \longrightarrow 0.$$
Since $\lm(W)<\infty$, for $n>>0$, $J^nM\cap W=0$.
% $$I^n\cap W=I^{n-n_0}(I^{n_0}\cap W)\subseteq I^{n-n_0}W\subseteq \m^{n-n_0}W=0.$$
We have for large $n$,
\begin{eqnarray*}
 H_{\ov J}(\ov M,n)&=&\lm(\ov M/J^n\ov M)\\&=&\lm(M/J^nM+W)\\
&=& \lm(M/J^nM)-\lm(J^nM+W/J^nM)\\
&=& \lm(M/J^nM)-\lm(W/J^nM\cap W)\\
&=& H_{J}(M,n)-\lm(W).
\end{eqnarray*}
Therefore $$P_{\ov J}(\ov M,n)=P_J(M,n)-\lm(W).$$
Hence $e_1(J,M)=e_1(\ov J, \ov M)$. Since $\ov M$ is Cohen-Macaulay, $e_1(\ov J,\ov M)=0$. Thus $e_1(J,M)=0$.
\end{proof}

\begin{example}
 Let $S=k[|X,Y,Z|]$ be a power series ring over a field $k$ and $J=(XZ,YZ,Z^2)$. 
Put  $R=S/J=k[[x,y,z]]$. Then $\dm R=2$ and $\depth R=0$. Consider the parameter ideal $I=(x,y).$  We calculate the Hilbert coefficients  of $I.$ 
Let $`-'$ denote the image in $\ov R=R/ H^0_{\m }(R)$ where $\m$ is the maximal ideal of $R.$ Notice that  for large $n,$
$$H^0_{\m}(R) = \dfrac{J: (X,Y,Z)^n}{J}= \dfrac{(J:X^n)\cap (J:X^{n-1}Y)\cap \ldots \cap (J:Y^n)}{J}
 = \dfrac{(Z)}{J}.$$
Therefore $R/H^0_{\m}(R)=\frac{k[|X,Y,Z|]}{J}/\frac{(Z)}{J}=k[|X,Y|]$ which is Cohen-Macaulay. Thus $e_1(x,y)=e_1(\ov x, \ov y)=0$. Notice that $e_2(\ov I)=e_2(I)-\lm(H^0_{\m}(R))$. Since $e_2(\ov I)=0$, $e_2(I)=\lm(H^0_{\m}(R))=1$.

\end{example}

\begin{example}
Let $S={\mathbb Q}[[x,y,z,u]]$, be the power series ring over $\mathbb Q$. 
Let $\phi:{\Q} [[x,y,z,u]]\longrightarrow {\Q}[|x,t|]$ defined by 
$$\phi(x)=x,\phi(y)=t^2 ,\phi(z)=t^5 \mbox{ and }  \phi(u)=t^7.$$ Then $$I_1 :=\ker \phi =(y^6-uz,z^3-y^4u,u-yz)$$ is a height 2 prime ideal. Let $\chi :\mathbb Q [[x,y,z,u]]\longrightarrow \mathbb Q [|u,t|]$ be defined by $$\chi(x)=t^2 ,\chi(y)=t^3 ,\chi(z)=t^4  \mbox{ and }\chi(u)=u.$$ Then $$I_2 :=\ker \chi=(y^2-xz,x^2-z)$$ is also a height 2 prime ideal. Put $I=I_1 \cap I_2$ and $R=S/I.$
%Let $I_1=(y^6-uz,z^3-y^4u,u-yz)$, $I_2=(y^2-xz,x^2-z)$ be two height 2 prime ideals of $S$. 
Then $\dim R=2$ and $R$ is not Cohen-Macaulay.  The ideal  $J=(x,u)R$ is  a parameter ideal in $R$  and $e_1(J)=-3$. 
 This example has been calculated using Cocoa. We thank M. Rossi for sending this CoCoA procedure to find Hilbert polynomial. The code is given below.\\ \\
 \texttt{Alias P:=\$contrib/primary;\\}
\texttt{Use S ::=} $\Q [x,y,z,u]$;\\
\texttt{I1  := Ideal}$(y^6-uz,z^3-y^4u,u-yz);$\\
%  [map is given by: y-->t^2, z-->t^5, u-->t^7, x-->x]
\texttt{I2 :=Ideal}$(y^2-xz,x^2-z);$\\% [map is given by: x-->t^2, y-->t^3, z-->t^4, u-->u]
\texttt{I := Intersection(I1, I2);\\
I;\\
Ideal}$(y^3z - xyz^2 - y^2u + xzu, x^2yz - yz^2 - x^2u + zu, x^2y^3u^2 - x^2z^4 - xyz^2u^2 + z^5 - y^2u^3 + xzu^3, y^5u^2 - y^2z^4 + xz^5 - yz^3u^2 - xy^2u^3 + z^2u^3, y^6u - y^2z^3u - xy^3u^2 - y^2z^3 + xz^4 + yz^2u^2, x^2y^4u - xy^2z^2u - x^2z^3 - y^3u^2 + xyzu^2 + z^4, x^2z^5 - x^2y^2u^3 - z^6 + y^2zu^3, y^7 - xyz^4 - xy^4u + xz^3u - y^2z^2 + xz^3, x^2y^6 - y^2z^4 - y^5u + yz^3u - x^2zu + z^2u, y^2z^5 - xz^6 - y^4u^3 + xy^2zu^3)$\\
\texttt{Q  := Ideal}$(x,u)+I;$\\
\texttt{Dim}$(S/Q);$\\
$0$\\
$J :=I1+I2;$\\
\texttt{Dim}$(S/J);$\\
$0$\\
\texttt{PS := P.PrimaryPoincare(I, Q); PS;}\\
$(12 - 3x) / (1-x)^2$\\
\texttt{Hilbert}$(S/J);$\\
%WARNING! HilbertPoincare input not homogeneous: computing LT...
$H(0) = 1, H(1) = 4, H(2) = 7, H(3) = 5, H(t) = 0$   for $t \geq  4.$
 
\end{example}

Recently the Negativity Conjecture has been settled in 
\cite{gghopv} for unmixed local rings. We generalize this to
finite unmixed modules over Cohen-Macaulay local rings.

\begin{definition}
 Let $(R,\m)$ be a Noetherian local ring of dimension $d.$
A finite $R-$module $M$ is called unmixed if for each associated prime $P$ of its $\m$-adic completion $\hat M,$ $\dim R/P=d.$ 
\end{definition}
  
We use Nagata's technique of idealization \cite{nag}. Let $M$ be an $R-$module. Let $R^*=R\oplus M$ be the direct sum of the $R-$modules $R$ and $M.$ Define multiplication in $R^*$ by
$$(r,m)((s,n)=(rs,rn+ms) \; \text{for all } r,s \in R; m,n \in M .$$  

In the next lemma we prove that the associated primes of the idealization $R^*$ come from those of $R$ and $M.$
\begin{lemma}\label{lema}
 Let $(R,\m)$ be a local ring and $M$ be a finite $R$-module. Let $A=R*M$ be the idealization of $M$ over $R$. Then
 $$\Ass A\subseteq \{P*M\mid P\in \Ass R \cup \Ass_R M\}.$$
 Moreover if $P\in \Ass_R M$ then $P*M\in \Ass A$.
\end{lemma}
\begin{proof}
% Since $(0*M)^2=0$, $0*M$ is a nilpotent ideal.
 Let $\mathcal P\in \Spec A$ then $\mathcal P\supseteq 0*M$ as $(0*M)^2=0$. Hence $\mathcal P/(0*M)\in \Spec (A/0*M)=\Spec R$. Therefore there exists a prime $P\in R$ such that ${\mathcal P}/{0*M}={P*M}/{0*M}$ which implies $\mathcal P=P*M.$ Thus every prime ideal of $A$ is of the form $P*M$ where $P$ is a prime ideal of $R$.
 
 Let $P*M\in \Ass A$ then $P=(0:(r,m))$, where $r\in R$ and $m\in M$. Let $a\in P$ then $(a,0)\in P*M$ which implies that $(ar,am)=(0,0)$. Thus $a\in (0:r)\cap  (0:m)$. Hence $P\subseteq (0:r)\cap (0:m)$. Let $b\in (0:r)\cap (0:m)$ then $(b,0)(r,m)=(0,0)$ which implies $(b,0)\in P*M$. Thus $b\in P$. Hence $P=(0:r)\cap (0:m)$. Therefore either $P=(0:r)$ or $P=(0:m)$. Hence $P\in \Ass R\cup \Ass_R M$. Therefore $\Ass A\subseteq \{P*M\mid P\in \Ass R \cup \Ass_R M\}.$\\
 Let $P\in \Ass_RM$ then $P=(0:m)$ where $m\in M$. Want to show that $P*M=(0:(0,m))$. Let $(a,n)\in P*M$. Since $(a,n)(0,m)=(0,am)=(0,0)$ therefore $(a,n)\in (0:(0,m))$. Conversely if  $(b,m')\in (0:(0,m))$ then  $b\in(0:m)$. Thus $P*M=(0:(0,m))$ and hence $P*M\in \Ass A$.
\end{proof}

\begin{theorem}
 Let $(R,\m)$ be a Cohen-Macaulay local ring and let $M$ be  an unmixed module with $\dm R=\dm M=d.$  If $e_1(J,M)=0$ for some parameter ideal $J$ for $M$. Then $M$ is is a Cohen-Macaulay $R-$module.
\end{theorem}
\begin{proof}
 Let $A=R*M$ be the idealization of $M$ over $R$. Then $\dm A=\dm R$. 
Note that $\widehat A=\widehat{R*M}=\widehat R *\widehat M$. If $P*\widehat M\in \Ass \widehat A$, then $P\in \Ass\widehat R\cup \Ass \widehat M$ by Lemma \ref{lema}. Since $R$ is Cohen-Macaulay and $M$ is unmixed $\dm \widehat R/Q=d$ for all $Q\in \Ass \widehat R\cup \Ass \widehat M.$ Therefore $\dm \widehat A/(P*\widehat M)=\dm \widehat R/P=d$. Hence $A$ is unmixed. Consider the  exact sequence of $R$-modules
 \begin{equation}\label{eq4}
 0\longrightarrow M\longrightarrow A\longrightarrow R\longrightarrow 0.
 \end{equation}
 Tensoring the above sequence with $R/J^n$ we get the following exact sequence
 $$0\longrightarrow M/J^nM\longrightarrow A/J^nA \longrightarrow R/J^n\longrightarrow 0.$$
Since the length function is additive, we get 
$$\lm(A/J^nA)=\lm(M/J^nM)+\lm(R/J^n).$$
Hence $P_J(A,n)=P_J(M,n)+P_J(R,n)$. Equating the coefficients of the Hilbert polynomials we get
$$e_1(J,A)=e_1(J,M)+e_1(J).$$
Since $R$ is Cohen-Macaulay $e_1(J)=0$. Thus $e_1(J,A)=0$. Hence by  \cite[Theorem 2.1]{gghopv} $A$ in Cohen-Macaulay ring. Applying depth lemma on the exact sequence (\ref{eq4}) we get that 
$$\depth M\geq \min \{\depth A,\depth R +1\}$$
which implies $\depth M=d$. Thus $M$ is Cohen-Macaulay.
\end{proof}

\section{Some Bounds for the Chern Number}

In this section we find an upper bound for the Chern number of an admissible filtration $\mathcal F.$ This bound yields
the Huckaba-Marley bound in Cohen-Macaulay case. We use Rees algebra of $\mathcal F$ and Serre's multiplicity formula in terms
of lengths of Koszul homology modules.

Let $A=\oplus_{n\geq 0}A_n$ be a standard graded algebra with $A_0=(R,\m)$ be a local ring. Let $M=\oplus_{n\geq 0}M_n$ be a finitely generated graded $A-$module of dimension $d$ such that $\lm(M_n)<\infty$ for all $n\geq 0$. Let $P_M(x)$ be the polynomial corresponding to the function $H_M(n)= \lm(M_n).$ Write
$$
P_M(x)=\sum_{i=0}^{d-1} (-1)^i e_i(M)\binom{x+d-i}{d-i}.
$$

\begin{lemma}\label{l1}
 Let $A=\oplus_{n\geq 0}A_n$ be a standard graded algebra with $A_0=(R,\m)$ be a local ring and let $M=\oplus_{n\geq 0}M_n$ be a finitely generated graded $A-$module of dimension $d$ such that $\lm(M_n)<\infty$ for all $n\geq 0$.  Then $$e_0(A_1,M)=e_0(M).$$
\end{lemma}

\begin{proof}
 Let $n_0$ be the largest degree of a homogeneous set of generators of $M$ as an $A$-module. Then
 $$M_{n_0+1}=A_{n_0+1}M_0+A_{n_0}M_1+\cdots +A_1M_{n_0}.$$
 Since $A$ is standard graded $A_r=(A_1)^r$ for all $r\geq 1$. Therefore we have
 $$M_{n_0+1}=(A_1)^{n_0+1}M_0+(A_1)^{n_0}M_1+\ldots +(A_1)M_{n_0}=A_1M_{n_0}.$$
 Hence for all $k\geq 1$, $M_{n_0+k}=(A_1)^kM_{n_0}$. Let $H(n)=\lm(M_n)$. Since
 $$\frac{M}{(A_1)^nM}=\dfrac{\oplus_{r\geq 0}M_r}{\oplus_{r\geq 0}A_1^nM_r}=M_0\oplus \cdots \oplus M_{n-1}\oplus\frac{M_n}{A_1^nM_0}\oplus \cdots \oplus \frac{M_{n+n_0}}{A_1^nM_{n_0}},$$
 we get $$\lm(M/(A_1)^nM)=\sum_{i=0}^{n-1}H(i)+\sum_{j=0}^{n_0}\lm\left( \frac{M_{n+j}}{A_1^nM_j}\right).$$
 Since the 2nd sum is a finite sum for large $n$ it is a polynomial function of degree at most $d-1$. Hence $\lm(M/A_1^nM)$ is a polynomial function of degree $d$ since $\sum_{i=0}^{n-1}H(i)$ is a polynomial function of degree $d$. Thus $e_0(A_1,M)=e_0(M)$.
\end{proof}

\begin{theorem}\label{thm2}
 Let $(R,\m)$ be a $d$-dimensional local ring and let $J$ be a parameter ideal of $R$. Let $\mathcal F=\{{J_n}\}$ be a $J-$admissible filtration. Let $A=R[Jt]=\oplus_{n\geq 0}J^nt^n$ and $B=\mathcal R(\mathcal F)=\oplus_{n\geq 0}{J_n}t^n$ and $M=B/A=\oplus_{n\geq 1}{J_n}/J^n$. 
If $\h (A:_A B)=1$ then 
 $$ e_1(\mathcal F)\leq  e_1(J)+\sum_{n\geq 1}\lm({J_n}/J{J_{n-1}}).$$
\end{theorem}
\begin{proof}
We may assume that $R$ is complete.
 Since $\mathcal F$ is an admissible filtration $B$ is a finitely generated $A$-module and hence $M$ is also a finitely generated $A$-module.
 Since $\h (A:_A B)=1$, $\dm M=d$. 
 Note that 
 \begin{eqnarray*}
 \lm(M_n)&=&\lm({J_n}/J^n)\\
 &=& \lm(R/J^n)-\lm(R/{J_n})\\
 &=& [ e_1(\mathcal F)-e_1(J)]{n+d-2\choose d-1}+ \mbox{ lower degree terms.}
  \end{eqnarray*}
Therefore $\lm(M_n)$ is a polynomial for large $n$ of degree $d-1$ with leading coefficient $ e_1(\mathcal F)-e_1(J)$. Note that 
$M/JtM=\oplus_{n\geq 1} {J_n}/J{J_{n-1}}$ and for large $n$, ${J_n}=J{J_{n-1}}$. Thus $\lm(M/JtM)<\infty$. By Lemma \ref{l1}, $\lm(M/J^nt^nM)$ is a polynomial 
for large $n$ of degree $d$ and $e_0(Jt,M)= e_1(\mathcal F)-e_1(J)$. By Serre's Theorem we have $$e_0(Jt,M)=\sum_{i=0}^d(-1)^i\lm(H_i(Jt,M))$$ where $H_i(Jt,M)$ is the $i^{th}$ 
 Koszul homology of $M$ with respect to $Jt.$  Note that $$H_0(Jt,M)=M/JtM={\displaystyle \bigoplus_{n\geq 1} {J_n}/J{J_{n-1}}}.$$ Let $\chi_1=\sum_{i=1}^d(-1)^{i+1}\lm(H_i(Jt,M))$. By \cite[Theorem 4.7.10]{bh}  
$\chi_1\geq 0$. Hence $$e_1(\mathcal F)-e_1(J)\leq \sum_{n\geq 1}\lm({J_n}/J{J_{n-1}}).$$ Thus we have $$ e_1(\mathcal F)\leq e_1(J)+\sum_{n\geq 1}\lm({J_n}/J{J_{n-1}}).$$
\end{proof}

\begin{corollary}
Let $(R,\m)$ be a $d$-dimensional analytically unramified local ring and let $J$ be a parameter ideal of $R$. Let $\mathcal F=\{\ov {J^n}\}$ denote the integral 
closure filtration. Let $A=R[Jt]=\oplus_{n\geq 0}J^nt^n$ and $B=\mathcal R(\mathcal F)=\oplus_{n\geq 0}\ov{J^n}t^n$ and $M=B/A=\oplus_{n\geq 1}\ov{J^n}/J^n$. 
If $\h (A:_A B)=1$ then 
 $$\ov e_1(J)\leq \sum_{n\geq 1}\lm(\ov{J^n}/J\ov{J^{n-1}})+e_1(J).$$
\end{corollary}

\begin{proof}
 Since $R$ is analytically unramified $\mathcal F=\{\ov{J^n}\}$ is a $J$-admissible filtration. Hence by Theorem \ref{thm2}, we have $$\ov e_1(J)\leq \sum_{n\geq 1}\lm(\ov{J^n}/J\ov{J^{n-1}})+e_1(J).$$
\end{proof}

\begin{corollary}[Huckaba-Marley]\cite[Theorem 4.7]{hm}
 Let $(R,\m)$ be a Cohen-Macaulay local ring of dimension $d,$ let $J$ be a parameter ideal of $R$ and let $\mathcal F =\{J_n\}$ be a $J$-admissible filtration. Then 
$${e_1}(\mathcal F)\leq \sum_{n\geq 1}\lm( {J_n}/J\ov{J_{n-1}}).$$
 \end{corollary}

\begin{proof}
 Since $R$ is Cohen-Macaulay $e_1(J)=0$. Hence by Theorem \ref{thm2}, we have 
$$ e_1(\mathcal F)\leq \sum_{n\geq 1}\lm({J_n}/J\ov{J_{n-1}}).$$
\end{proof}

% \begin{corollary}
%  Let $(R,\m)$ be a $d$-dimensional Buchsbaum local ring and $J$ be a parameter ideal of $R$ 
% \end{corollary}

\section{ Some Further Estimates for  the Chern Number}

In this section, we provide some estimates for  the Chern number in terms of a cover $S/R$ of finite length such that $S$ is local.  

 Let $(R,\m)$ be a Noetherian local ring of dimension $d\geq 1,$ and let $S/R$ be a  cover of finite length such that $S$ is local. Let $\n$ be the maximal ideal of $S,$ let $\rho=[S/\n:R/\m],$ and let $\nu=\lm(S/R).$  Let $J$ be a parameter ideal of $R.$ Then $JS$ is a parameter ideal of $S.$ 

Let $P_{J}(X)$ and $P_{JS}(X)$ be the polynomials associated to the functions 
$n\mapsto \lm (R/J^{n+1}) $ and  $n\mapsto \lm_S(S/J^{n+1}S),$ respectively. 

For a finitely generated $R$-module $M,$ let $\mu_R(M)$ denote the minimum number of generators of $M.$ 

\begin{proposition} {\rm (1)} For every $n\geq 1$ we have 
\begin{eqnarray*}
\mu_R(S/R) {n+d-1\choose d-1} &  \leq &  \lm (S/(R+JS)){n+d-1\choose d-1} \\ 
			      &  \leq &  \lm (J^{n}S/J^{n}) \\   
			      &  \leq &  \lm (S/R){n+d-1\choose d-1}.
\end{eqnarray*}

{\rm (2)}    The function $n\mapsto \lm (J^{n+1}S/J^{n+1}) $ is of polynomial  type with associated polynomial 
$ P_J(X) + \nu  - \rho P_{JS}(X),$ and further, 
$$
 P_J(X) + \nu  - \rho P_{JS}(X) = - e_1(J){X+d \choose d-1} + f(X) \hbox{ with  } \deg f(X)\leq d-2. 
$$ 

{\rm (3)} $e_0(J)= \rho e_o(JS).$ 

{\rm (4)} $\mu_R(S/R)\, \leq \, \lm (S/(R+JS)) \,\leq \, -e_1(J) \,\leq \, \lm (S/R).$ 

{\rm (5)} If $\mu_R(S/R) = \lm (S/R)$ (equivalently, if $\m S\sbs R$) then 
$$ 
e_1(J) = -\mu_R(S/R) = -\lm (S/R) 
$$ 
and 
$$ 
\lm (J^{n}S/J^{n})= -e_1(J){n+d-1 \choose d-1} \;\; \hbox{ for every } n\geq 1. 
$$ 
\end{proposition}

\begin{proof} (1) The first inequality holds trivially because 
$$ \mu_R(S/R) = \mu_R(S/(R+JS)) \leq \lm (S/(R+JS)).$$ 
To prove the second inequality, let $m = \lm (S/(R+JS)),$ and choose $y_1,\ldots, y_m\in S$ such that if $M_i=R+JS+(y_1,\ldots, y_i)R$ then  $S=M_m$ and  $\lm (M_{i}/M_{i-1})=1$  for every $i.$

 Let $J=(x_1,\ldots, x_d)R.$ For a fixed $n,$  let $s={n+d-1\choose d-1},$ and let 
$\alpha_1,\ldots, \alpha_s$ be all the monomials 
of degree $n$ in $x_1,\ldots, x_d.$ Then $J^n=(\alpha_1,\ldots, \alpha_s)R.$ We have
to show that $ms\leq \lm (J^nS/J^n).$ 

Since $S=R+JS+(y_1,\ldots, y_m)R,$ we have $J^nS= J^n+ J^{n+1}S+J^n(y_1,\ldots, y_m)R.$ Let 
$N_i=J^n+ J^{n+1}S+J^n(y_1,\ldots, y_i)R.$ Then $N_0= J^n+ J^{n+1}S$ and  $N_m=J^nS,$ and we have the sequence
$N_0\sbs N_1\sbs \cdots \sbs  N_m.$ So it is enough to prove that $\lm (N_i/N_{i-1}) \geq s$ for every $i\geq 1.$ 

 For a fixed $i\geq 1$ and for $0\leq j\leq s,$ let $P_j=N_{i-1} +(\alpha_1,\ldots, \alpha_j)y_i.$  Then $P_0= N_{i-1}$ and $P_s= N_i$ and we have the sequence   $P_0\sbs P_1\sbs \cdots \sbs P_s.$
So it is enough to prove that all the inclusions in this sequence are proper. 

Suppose, to the contrary, that $P_{j}=P_{j+1}$ for some $j\leq s-1.$ Then 
$$\alpha_{j+1} y_i \in P_{j} = J^n+ J^{n+1}S+J^n(y_1,\ldots, y_{i-1})+(\alpha_1,\ldots, \alpha_{j})y_i. 
$$ 
So we can write 
$$\alpha_{j+1}y_i =  \beta +  \sum_{k=1}^s a_k\alpha_k +\sum_{k=1}^s b_k\alpha_k + \sum_{k=1}^j c_k y_i \alpha_k
$$
with $\beta\in J^{n+1}S,\; a_k,c_k \in R $ and $b_k\in (y_1,\ldots y_{i-1})R.$ Since $JS$ is a parameter ideal in the  Cohen-Macaulay local ring $S,$  gr$_{JS}(S)$ is a polynomial ring in the images of $x_1,\ldots, x_d.$ Therefore, since $\alpha_1, \ldots, \alpha_s$ are distinct monomials in $x_1,\ldots, x_d,$  the coefficient of each $\alpha_k$ on the two sides of the above equality are congruent modulo $JS.$ In particular, looking at the coefficient of $\alpha_{j+1}$, we get 
$y_i\in R+JS+(y_1,\ldots, y_{i-1})R.$ 
This contradicts the condition  $\lm (M_i/M_{i-1})=1,$ so the second inequality of (1) is proved.    

To prove the third inequality, choose a sequence 
$$ 
R=M_0\proset M_1  \proset \cdots \proset  M_{\nu}=S 
$$ 
of $R$-submodules such that $M_{i}/M_{i-1} \iso R/\m$ for every $i\geq 1.$  
Then each $M_{i}=Rz_i+M_{i-1}$ for some $z_i\in S$ such that $\m z_i\sbs M_{i-1}.$ 
For a fixed $n,$ we have $J^{n}=(\alpha_1,\ldots, \alpha_s)R$ as above. Therefore   
$$J^{n}M_{i}=J^{n}z_i+J^{n}M_{i-1} = (\alpha_1z_i, \ldots , \alpha_sz_i)+J^{n}M_{i-1}.$$ 
Further, $\m \alpha_jz_i\sbs M_{i-1}\alpha_j \sbs J^{n}M_{i-1}.$ 
Therefore $\lm  (J^{n}M_{i}/J^{n}M_{i-1})\leq s$ for every $i\geq 1.$ Now, the sequence  
$$ 
J^{n}=J^{n}M_0\sbs  J^{n}M_1  \sbs \cdots \sbs  J^{n}M_{\nu}=J^{n}S 
$$ 
shows that 
$ \lm  (J^{n}S/J^{n})\leq \nu  s = \lm (S/R) {n+d-1\choose d-1}.$  

This completes the proof of (1).

(2) From the commutative diagram 
$$
\xymatrix{
J^{n+1}S \ar[r] & S \\ 
J^{n+1} \ar[r] \ar[u] & R \ar[u]    
}
$$
of inclusions, we get 

\begin{eqnarray*}
\lm (J^{n+1}S/J^{n+1})& = &\lm (R/J^{n+1}) + \nu  - \lm (S/J^{n+1}S) \\
&=&\lm (R/J^{n+1}) + \nu  - \rho \lm_S(S/J^{n+1}S).
\end{eqnarray*}

Therefore the function $n\mapsto \lm (J^{n+1}S/J^{n+1}) $ is of polynomial  type with associated polynomial 
$$ 
Q(X) :=  P_J(X) + \nu  - \rho P_{JS}(X).
$$
Since this function is squeezed between two polynomial functions of the same degree $d-1$  appearing in (1), we get   
$$ 
Q(X) = e{X+d-1 \choose d-1} + f(X)   \eqno{(*)} 
$$ 
with 
$\lm (S/(R+JS)) \leq e \leq \lm (S/R)$ and  $\deg f(X) \leq d-2.$ Since $JS$ is a parameter ideal in the Cohen-Macaulay local ring $S,$ we have 
$$ 
P_{JS}(X) =e_0(JS){X+d\choose d}. 
$$ 
Substituting the above expressions  for $P_{JS}(X)$ and  $Q(X)$ in  the formula 
$$ 
Q(X) =  P_J(X) + \nu  - \rho P_{JS}(X), 
$$
we get 
$$ 
Q(X) = - e_1(J){X+d-1 \choose d-1} + f(X) 
$$ 
with  $\deg f(X) \leq d-2,$ as required.

(3) We have $\deg P_{J}(X)  = d = \deg P_{JS}(X).$ Therefore, since $\deg(P_J(X)- \rho P_{JS}(X))\leq d-1$ by (2), we get 
$e_0(J)= \rho e_o(JS).$   

(4) This is immediate from (1) and (2).   

(5) This is immediate from (1) and (4).   
\end{proof}

\begin{corollary} In the above set up, assume further that $S$ is Cohen-Macaulay. If $\mu_R (S/R)= \lm  (S/R)$ (equivalently, if $\m S\sbs R$) then 
$$e_1(J)=- \mu_R (S/R)= - \lm  (S/R)$$ 
for every parameter ideal $J$ of $R.$ Further, in this case $R$ is Buchsbaum. 
\end{corollary}

\begin{proof}  The first part is immediate from the above proposition.  Taking $n=0$ in the commutative square appearing in the above proof, we get 
$$ 
\lm  (S/R) + \lm  (R/J) =  \lm  (S/JS) + \lm  (JS/J). 
$$ Since $JS$ is a parameter ideal in the Cohen-Macaulay local ring $S,$ we have $\lm_S(S/JS)=e_0(JS).$ Therefore $\lm  (S/JS)= \rho e_0(JS) = e_0(J)$ by the above proposition. Further, taking $n=1$ in part (5) of the above proposition, we get
$$\lm  (JS/J)=-e_1(J) d = \lm(S/R) d= \nu d.$$ 
Substituting  these values in the formula displayed above, we get  $\lm  (R/J) - e_0(J)= (d-1)\nu.$ Thus $\lm  (R/J) - e_0(J)$ is independent of the parameter ideal $J,$ so  $R$ is Buchsbaum. 
\end{proof}

\begin{example} These are examples to show that for $d=2$ the Chern number $e_1(J)$  can attain every value  in the range given by Proposition (6.1). More precisely,  given any integers $r,\,p  $ with $1\leq r\leq p  ,$ there exists a Noetherian local ring $R$ of dimension $2$, a Cohen-Macaulay cover $S/R$ of finite length and a parameter ideal $J$ of $R$ such that $\mu_R (S/R)=1,\; \lm  (S/R)=p  \;$ and $e_1(J)=-r.$ In fact, it can be verified by a direct computation that these equalities hold in the following situation: 
 $R=k[[t^2,t^3, x, tx^p  \,]] \sbs S=k[[t,x]]$ and $J=(t^2,x^r)R,$ where $k$ is a field  and $t$ and $x$  are indeterminates. 

\end{example}

 \end{document}